\documentclass[reqno]{amsart}
\usepackage{amsmath,amsfonts,amsthm}
\usepackage{color}
\usepackage[top=3.8cm, bottom=3cm, left=4.47cm, right=4.47cm]{geometry}

\numberwithin{equation}{section}

\theoremstyle{plain}
\newtheorem{theorem}{Theorem}[section]
\newtheorem{lemma}[theorem]{Lemma}

\newtheorem{corollary}[theorem]{Corollary}

\theoremstyle{remark}

\DeclareMathOperator{\proj}{proj}
\DeclareMathOperator{\n}{\text{\bf{\emph n}}}

\newcommand{\av}{-\kern-10.7pt\int} 
\newcommand{\R}{\mathbb{R}}
\newcommand{\dd}{\mathrm{d}}

\usepackage{hyperref}

\begin{document}

\title[The LSW equation for reaction-controlled kinetics]{The Lifshitz--Slyozov--Wagner equation\\ for reaction-controlled kinetics}
\author{Apostolos Damialis}
\address{Institut f\"ur Mathematik\\ Humboldt-Universit\"at zu Berlin\\ Unter den Linden 6\\ 10099 Berlin\\ Germany} 
\email{\href{mailto:damialis@mathematik.hu-berlin.de}{\texttt{damialis@mathematik.hu-berlin.de}}\bigskip} 
\curraddr{Department of Mathematics\\ University of Athens\\ Panepistemiopolis\\ 15784 Athens\\ Greece}
\thanks{This work was supported by the DFG through the Gra\-du\-ier\-ten\-kol\-leg RTG-1128 ``Analysis, Numerics, and Optimization of Multiphase Problems'' at the Humboldt-Universit\"at zu Berlin.}

\begin{abstract}
We rigorously derive a weak form of the Lifshitz--Slyozov--Wagner equation as the homogenization limit of a Stefan-type problem describing reaction-controlled coarsening of a large number of small spherical particles. Moreover, we deduce that the effective mean-field description holds true in the particular limit of vanishing surface-area density of particles.
\end{abstract}

\maketitle

\section{Introduction}
The late-stage behavior of a material undergoing a first-order phase transition (due to changes in temperature and/or pressure for example) is characterized by thermodynamic instability resolved through phase separation and consequent coarsening of the emerging phase. In the case of the new phase occupying much smaller volume fraction, and thus ap\-pearing as well-separated particles, this coarsening process (known as Ostwald ripening) is driven by the minimization of surface energy at the interface via diffusional mass exchange between particles while the total mass or volume of each phase is conserved. The result of this kind of mass diffusion from regions of high to regions of low interfacial curvature is the growth of large particles and the shrinkage and final extinction of smaller ones. For a review of some aspects of Ostwald ripening, mainly from the physical and modeling viewpoint, see the survey by Voorhees \cite{voorhees} or the book by Ratke and Voorhees \cite{voorhees_book}. 

In this coarsening scenario the mass-diffusion process can be controlled by two different mechanisms: either by the diffusion of atoms away from the particles and into the bulk, 
or by the reaction-rate of attachment of atoms at the phase interface. In the former case (diffusion control), the random exchange of atoms between the particles and the bulk is sufficiently rapid and the surrounding of each particle is in thermal equilibrium with the atoms in it; in the latter (interface-reaction control), detachment and attachment are slow compared to diffusion and the surrounding bulk can be out of equilibrium with the particle interface. We refer to the physics literature for more details, for example, Slezov and Sagalovich \cite{slezov}, Bartelt, Theis, and Tromp \cite{bartelt}; for a related mathematical treatment see Dai and Pego \cite{daipego}.

The classical theory for Ostwald ripening was developed by Lifshitz and Slyozov \cite{ls} and Wagner \cite{wagner} in the case of supersaturated solid solutions in three dimensions. The Lifshitz--Slyozov--Wagner theory statistically characterizes the evolution by the particle-radius density $n(t,R)$, where $n(t,R)\, \dd R$ is defined to be the number of particles with radii between $R$ and $\dd R$ at time $t$ per unit volume. In the late stages of the phase tran\-si\-tion nucleation and coalesence of particles can be neglected since new nuclei dissolve immediately and since particles cannot merge because of the large distances between them. Thus, the particle-radius density satisfies the continuity equation (see \cite[\S 5.1]{voorhees_book})
\[
\frac{\partial}{\partial t} n(t,R) + \frac{\partial}{\partial R} \big( \text{\sl v}\kern.5pt (t,R) n(t,R) \big) = 0,
\]
where $\text{\sl v}\kern.5pt (t,R)$ denotes the growth rate of particles of radius $R$ at time $t$.
Using a mean-field ansatz (cf.\ Section~\ref{growth}), Lifshitz, Slyozov, and Wagner 
formally calculate that 
\[
\frac{\partial}{\partial t} n(t,R) + \frac{\partial}{\partial R} \left( \frac{1}{R^2}(R\bar{u} -1) n(t,R) \right) = 0,
\]
with
\[
\bar{u}(t) = \int_{0}^{\infty} n(t,R)\, \dd R \,\bigg/ \int_{0}^{\infty} Rn(t,R)\, \dd R,
\]
in the diffusion-controlled case, and
\[
\frac{\partial}{\partial t} n(t,R) + \frac{\partial}{\partial R} \left( \left(\bar{u} - \frac{1}{R}\right) n(t,R) \right) = 0,
\]
with
\[
\bar{u}(t) = \int_{0}^{\infty} Rn(t,R)\, \dd R \,\bigg/ \int_{0}^{\infty} R^2n(t,R)\, \dd R,
\]
in the reaction-controlled one, both results valid in the limit of vanishing mass or volume fraction of particles.

In \cite{niethammer1} and \cite{niethammer} Niethammer rigorously derived the effective equations in the dif\-fu\-sion-controlled case, starting from a quasi-static one-phase Stefan problem with surface tension and kinetic undercooling, 
\begin{equation}\label{stefan}
\left.
\begin{array}{l}
-\Delta u = 0\quad \text{in } \Omega \setminus G,\medskip\\
V = \nabla u \cdot \n\quad \text{on } \partial G, \medskip\\
u = H + \beta V\quad \text{on } \partial G,
\end{array}\right\}
\end{equation}
and restricting it to spherical particles. The same was also done in \cite{niethammer1} for the full time-dependent parabolic problem but without the kinetic-drag term $\beta V$. Here, $u$ is a chemical potential, $\n$ is the outer normal to the particle phase $G$, $V$ is 
the normal velocity of the phase interface $\partial G$, and $H$ is its mean curvature. The domain $\Omega \subset \R^3$ is considered bounded and $\beta$ is a parameter that comes from the nondimensionalization and scales like diffusivity over mobility. The second boundary condition is the Gibbs--Thomson law, coupling the curvature of the interface with the chemical potential, modified by accounting for kinetic drag. Note that while under diffusion control the parameter $\beta$ is small and the kinetic drag can even be neglected (thus yielding the well-known Mullins--Sekerka model \cite{ms}), in the reaction-controlled case the values of $\beta$ are large and, therefore, the kinetic-drag term is necessary. For a derivation of such sharp-interface free-boundary problems from continuum mechanics and thermodynamics see the book of Gurtin \cite{gurtin}. 

The goal in the following is to use the techniques developed in \cite{niethammer1} and  \cite{niethammer} to derive the effective equations in the reaction-controlled case. This involves passing over to a different time scale incorporating the parameter $\beta$ tending to infinity (see Section~\ref{formulation}) and, as a result, some extra manipulations in the proofs. Except for the scaling, in Section~\ref{formulation} we also give short proofs of some useful preliminaries and discuss the validity of the mean-field description while in Section~\ref{growth} we prove pointwise estimates for approximate solutions and for the growth rates of particles. Finally, using these estimates, in Section~\ref{homogenization} we pass to the homogenization limit of infinitely-many particles and obtain a weak form of the Lifshitz--Slyozov--Wagner equation.

In comparison with the results in the diffusion-controlled case, we make precise that the crucial quantity that has to vanish in order to neglect direct interactions between particles and justify the expected mean-field law is the surface-area density of the particles in contrast to their capacity in the other case (see \cite{niethammer1} and \cite{niethammer}). This difference is of interest since the asymptotic limits of vanishing surface area and capacity have different physical interpretations and further refine the na\"{\i}ve general limit of vanishing mass or volume. For the reaction-controlled case though, the result is in some sense to be expected since the limit of vanishing surface-area density corresponds to the physics of the interface-reaction-controlled scenario, where there is an obvious dependence on the area of the interface.

\section{Formulation, scaling, and preliminary estimates}\label{formulation}
We start with problem \eqref{stefan} where the quasi-static approximation to the parabolic diffusion equation is justified by the small interfacial velocities present during late-stage coarsening. (See the discussion in Mullins and Sekerka \cite{ms}.)

We further suppose that the solid phase consists of spherical particles with centers fixed in space, a simplification that can be justified by the work of Alikakos and Fusco \cite{alikakos1}, \cite{alikakos2}, and Vel\'azquez \cite{velazquez}. Denoting these particles as $B_i$, where each $B_i$ is the closed ball $\overline{B(x_i, R_i(t))},$ the particle phase is then the union $\cup B_i$ and its isotropic evolution can be modeled by averaging the flux in the Stefan condition, i.e., 
\[
V = \dot{R}_i(t) := \av_{\partial B_i} \nabla u \cdot \n,
\]
where the average integral is defined as
\[
\av_D f := \frac{1}{|D|} \int_D f,
\]
for a function $f$ on some domain $D$, and where the overdot denotes a derivative with respect to time; the Gibbs--Thom\-son law becomes then
\[
u = \frac{1}{R_i} + \beta \dot{R}_i,
\]
since in the case of spheres the mean curvature is the inverse radius.

To have many small particles in a bounded domain, for a system with size of
order $\mathrm{O}(1)$, say the unit cube $[0,1]^3$,  let $\delta$ be the typical
particle distance with $0 < \delta \ll 1$. For the distribution of particle
centers in space, we assume, for simplicity, that they are situated on a
three-dimensional lattice of spacing $\delta$. Then, the initial number density
of particles $N_{\mathrm{i}}(\delta)$ will be bounded by $1 / \delta^3$, and
for the particles to be small let the typical particle size be $\delta^\alpha$
for $\alpha > 1$. For times $t \in [0,T]$ we choose a $\delta$ small enough so
that adjacent particles of size $\delta^\alpha$ will not collide during the
evolution up to a maximal time $T$.

Concerning the assumption on the spatial distribution of particles, a more general assumption like $\inf_{i \neq j} |x_i - x_j| > c \delta$, for a constant $c>0$, would still be enough for our purposes in this work. These considerations will also be used in the proof of Lemma \ref{u_bound} where we approximate a certain sum over all particles by an integral. For an approach using more sophisticated deterministic and stochastic assumptions on the distribution of particles with respect to homogenization we refer to Niethammer and Vel\'azquez \cite{nietvelaz1}, \cite{nietvelaz2}, where also further refinements of the theory are made.

To have particle sizes of order $\mathrm{O}(1)$ as well, we rescale 
\[
R_{i}^{\delta} := \frac{R_i}{\delta^\alpha},
\]
and motivated by the scaling invariance of problem \eqref{stefan} (cf.\ \cite{daipego}),
\[
u^\delta := \delta^\alpha u,\quad t^\delta := \frac{t}{\delta^{2\alpha}}.
\]
Notice that this rescaling is another way of addressing the reaction-controlled regime. Instead of rescaling time by $\beta$ and then letting $\beta$ tend to infinity, we keep $\beta$ fixed and positive, and specially rescale as above letting $\delta$ tend to zero. Since now $\beta$ plays no significant role, we will set it to unity in what follows. In addition, one easily sees that the transformations $R_{i}^{\delta}$, $u^\delta$, and $t^\delta$ preserve the form of the equations. From hereon we also drop the superscript $\delta$ from the notation for time and to denote the dependence on the new scale we write 
\[
B_{i}^{\delta} := \overline{B( x_i,\delta^\alpha R_{i}^{\delta})}.
\]
Finally, note that under diffusion control the relevant scale for time would be $\delta^{3\alpha}$ instead of $\delta^{2\alpha}$. This difference is key to all that follows, leading to different considerations on the validity of the mean-field model. (Cf.\ the remarks following Lemma \ref{a_priori}.)

As initial data, for every particle-center $x_i$ we associate a corresponding bounded initial radius $R_{i}^{\delta}(0)$ with the assumption that 
\[
{\textstyle \sup_{i \in N_{\mathrm{i}}}} R_{i}^{\delta}(0) \leq R_0,
\]
uniformly for some constant $R_0$. To consider a closed system, we impose a no-flux Neumann boundary condition on the outer boundary of $\Omega$, i.e., 
\[
\nabla u^\delta \cdot \n = 0\quad \text{on } \partial \Omega.
\]

In case the $i$th particle vanishes at time $t_i := \sup \{ t \mid R_{i}^{\delta}(t) > 0 \},$ for times later than $t_{i}$ we define $R_{i}^{\delta}$ to be zero, reduce the number $N(t) := \{ j \mid R_{j}^{\delta} (t) >0 \}$ of active particles by one, and neglect the boundary $\partial B_{i}^{\delta}$ in the boundary conditions. In the following, all sums, unions, and suprema will run over the set $N(t)$, with $N(0) \equiv N_{\mathrm{i}}$, and any further reference to the particle-number density will mean the active particle-number density $N$ unless otherwise noted.

Summarizing, the restricted and rescaled problem for the particle radii can be considered as a nonlocal, $N$-dimensional system of ordinary differential equations
\begin{equation}
\dot{R}_{i}^{\delta}(t) = \frac{1}{4 \pi \delta^{2\alpha} R_{i}^{\delta}(t)^2} \int_{\partial B_{i}^{\delta}(t)} \nabla u^\delta \cdot \n\quad \text{on } \partial B_{i}^{\delta}(t), \label{p1}
\end{equation}
for times $t \in (0,t_{i})$, $t_i<T$, and with bounded initial data $R_{i}^{\delta}(0)$ for every $i$, while the chemical potential is determined by
\begin{equation}
-\Delta u^\delta(t,x) = 0\quad  \text{in } \Omega \setminus \cup B_{i}^{\delta} (t), \label{p2}
\end{equation}
\begin{equation}
u^\delta(t,x) = \frac{1}{R_{i}^{\delta}(t)} + \dot{R}_{i}^{\delta}(t)\quad \text{on } \partial B_{i}^{\delta}(t), \label{p3}
\end{equation}
and the Neumann condition on the outer boundary.

Global existence and uniqueness of continuous, piecewise-smooth solutions for a similar restricted Stefan problem was proved in \cite{niethammer} by an application of the Picard--Lindel\"of theorem, the only difference being the different time scale. These solutions are not globally smooth due to the singularities arising from the extinction of particles; however, they are smooth in the intervals between the extinction times $t_i$. In the following, when we mention solutions of the problem we will mean such continuous, piecewise-smooth solutions that exist up to any given time $T$.

It is easy to see that equations \eqref{p1}, \eqref{p2}, \eqref{p3}, along with the outer boundary condition conserve the volume and decrease the interfacial area of the particle phase. Indeed, differentiating the total volume of particles with respect to time gives 
\[ \frac{\dd}{\dd t} \sum_i R_{i}^{\delta}(t)^3 = 3\sum_i R_{i}^{\delta}(t)^2 \dot{R}_{i}^{\delta} (t) = 3 \sum_i R_{i}^{\delta}(t)^2 \frac{1}{4 \pi \delta^{2\alpha} R_{i}^{\delta} (t)^2} \int_{\partial B_{i}^{\delta}} \nabla u^\delta \cdot \n \]
where the last sum vanishes due to the divergence theorem, equation \eqref{p2}, and the no-flux condition on $\partial \Omega$. The decrease of total surface area follows from the next {\em a priori} estimate.
\begin{lemma}\label{a_priori}
For any time $t\in(0,T)$, the solutions of the problem satisfy the following energy equality.
\[
\sum_i \int_{0}^{t} (R_{i}^{\delta})^2 |\dot{R}_{i}^{\delta}|^2 + \frac{1}{2} \sum_i R_{i}^{\delta}(t)^2 + \frac{1}{4 \pi \delta^{2\alpha}} \int_{0}^{t} \kern-4pt \int_{\Omega \setminus \cup B_{i}^{\delta}} |\nabla u^\delta|^2 = \frac{1}{2} \sum_i R_{i}^{\delta}(0)^2.
\]
\end{lemma}
\begin{proof}
Multiplying $-\Delta u^\delta = 0$ with $u^\delta$, integrating over 
$\Omega \setminus \cup B_{i}^{\delta}$, and integrating by parts gives 
\[
\int_{\Omega \setminus \cup B_{i}^{\delta}} |\nabla u^\delta|^2 + \sum_i \int_{\partial B_{i}^{\delta}} (\nabla u^\delta \cdot \n) u^\delta - \int_{\partial \Omega} (\nabla u^\delta \cdot \n)u^\delta = 0,
\]
where the last term vanishes due to the Neumann condition on the outer boundary. 
Thus, using equations \eqref{p3} and \eqref{p1} we get 
\begin{align*}
-\int_{\Omega \setminus \cup B_{i}^{\delta}} |\nabla u^\delta|^2 
& = \sum_i \left( \frac{1}{R_{i}^{\delta}} + \dot{R}_{i}^{\delta} \right) \int_{\partial B_{i}^{\delta}} \nabla u^\delta \cdot \n\\
& = \sum_i \left( \frac{1}{R_{i}^{\delta}} + \dot{R}_{i}^{\delta} \right) 4 \pi \delta^{2\alpha} (R_{i}^{\delta})^2 \dot{R}_{i}^{\delta} \\
& = 4 \pi \delta^{2 \alpha} \sum_i \big( R_{i}^{\delta} \dot{R}_{i}^{\delta} + (R_{i}^{\delta})^2 |\dot{R}_{i}^{\delta}|^2\big), 
\end{align*}
and after rearranging,
\begin{equation}\label{rr}
\sum_i (R_{i}^{\delta})^2 |\dot{R}_{i}^{\delta}|^2 + \sum_i R_{i}^{\delta} \dot{R}_{i}^{\delta} + \frac{1}{4 \pi \delta^{2\alpha}} \int_{\Omega \setminus \cup B_{i}^{\delta}} |\nabla u^\delta |^2 = 0.
\end{equation}
The result follows from an integration over time.
\end{proof}

After normalization with respect to the initial particle-number density $N_{\mathrm{i}}$, this energy equality can yield useful information on the validity of the mean-field approach. In fact, we have
\[
\frac{1}{N_{\mathrm{i}}} \sum_i \int_{0}^{t} (R_{i}^{\delta})^2 |\dot{R}_{i}^{\delta}|^2 + \frac{1}{2N_{\mathrm{i}}} \sum_i R_{i}^{\delta}(t)^2 + \frac{1}{4 \pi N_{\mathrm{i}} \delta^{2\alpha}} \int_{0}^{t} \kern-4pt \int_{\Omega \setminus \cup B_{i}^{\delta}} |\nabla u^\delta|^2 = 
\frac{1}{2N_{\mathrm{i}}} \sum_i R_{i}^{\delta}(0)^2,
\]
where the right-hand side is uniformly bounded by the assumption on the initial
radii. For the left-hand side to stay bounded as well, if the quantity
$N_{\mathrm{i}} \delta^{2\alpha}$ tends to zero, the same must hold for $|\nabla
u^\delta|$ and it is exactly this limit of vanishing surface-area density of
particles that results in a mean field that is constant in space since, in
particular, 
\[
\nabla u^\delta \to 0 \text{ in } L^2 \big(0,T;H^1(\Omega)\big).
\]
Here and in the following, to obtain global estimates that are uniform in
$\delta$ we extend $u^\delta$ to the interior of particles, and thus to the
whole of $\Omega$, by its boundary values. It is important to note that in our
scaling setup, for the surface area to vanish as $\delta$ tends to zero, the
exponent $\alpha$ must be strictly larger than $3/2$ since $N_{\mathrm{i}}$
is $\mathrm{O}(1 / \delta^3)$. These facts will be made precise in
Corollary \ref{uu_bound} where we give an estimate of the mean-field effect.
Note also that we do not address here the critical case $\alpha= 3 / 2$
that corresponds to finite surface area. For that one would have to use the
different methods developed by Niethammer and Otto in \cite{niet-otto}.

Finally, note that for similar considerations under diffusion control, the corresponding quantity would be the capacity $N_{\mathrm{i}} \delta^{\alpha}$ due to the different time scale. In three dimensions, this capacity effect fits to general homogenization results as in the work of Cioranescu and Murat \cite{murat}; to our knowledge though, the surface-area effect has not been explicitly discussed in the relevant literature.

\section{Approximation and growth-rate estimates}\label{growth} 
As in the mean-field ansatz of Lifshitz, Slyozov, and Wagner, we suppose that the system is dilute enough so that particles behave as if they were isolated and we base our approximation on the solution of a single-particle problem. 

Consider problem \eqref{p1}, \eqref{p2}, \eqref{p3} for a single spherical particle centered at the origin and with initially unscaled radius $r$ that we rescale as $r^\delta := r / \delta^\alpha$, along with the corresponding reaction-controlled rescalings for a chemical potential $u_{r}^{\delta}$ and time, as in Section \ref{formulation}. For this rescaled particle $B_{r}^{\delta}$ we consider the following problem in the whole space:
\[
\dot{r}^\delta(t) = \frac{1}{4 \pi \delta^{2\alpha} r^\delta(t)^2} \int_{\partial B_{r}^{\delta}} \nabla u_{r}^{\delta} \cdot \n\quad \text{on } \partial B_{r}^{\delta},
\]
where the chemical potential $u_{r}^{\delta}(t,x)$ satisfies
\[
- \Delta u_{r}^{\delta}(t,x) = 0\quad \text{in } \R^3 \setminus B_{r}^{\delta},
\]
\[
u_{r}^{\delta}(t,x) = \frac{1}{r^\delta(t)} + \dot{r}^\delta(t)\quad \text{for } x \in \partial B_{r}^{\delta},
\]
and the mean-field assumption is posed as a condition at infinity, i.e.,  
\[
\lim_{|x| \to \infty} u_{r}^{\delta} (t,x) = \bar{u}_{r}^{\delta}(t).
\]
This problem can be explicitly solved to give
\[
u_{r}^{\delta}(t,x) = \bar{u}_{r}^{\delta}(t) + \frac{\delta^\alpha r^\delta(t)}{1 + \delta^\alpha r^\delta(t)} \big(1 - \bar{u}_{r}^{\delta}(t)r^\delta(t)\big) \frac{\delta^\alpha}{|x|}
\]
and
\[
\dot{r}^\delta(t) = \frac{1}{1 + \delta^\alpha r^\delta(t)}
\left( \bar{u}_{r}^{\delta}(t) - \frac{1}{r^\delta(t)} \right).
\]
Note that in the formal limit of $\delta$ tending to zero, the expected effective equations take the general form
\[
u(t,x) = \bar{u}(t) ~~\text{ and }~~ \dot{r} = \bar{u} - \frac{1}{r},
\]
as in the reaction-controlled Lifshitz--Slyozov--Wagner theory.

Going now back to the many-particle problem, a calculation using the single-particle growth rate above along with the requirement that the volume is conserved gives the following expression for the mean field
\begin{equation}\label{meanfield}
\bar{u}^{\delta} =  \sum_{i} \frac{R_{i}^{\delta}}{1 + \delta^\alpha R_{i}^{\delta}} \bigg/ \sum_{i} \frac{(R_{i}^{\delta})^2}{1 + \delta^\alpha R_{i}^{\delta}}.
\end{equation}
The effect of this mean field plus a sum of single-particle solutions will be the monopole approximation to the solution $u^\delta$ supposing that there are no direct interactions between particles. To this end, let us define the approximate solution
\begin{equation}
\zeta^\delta (t,x) := \bar{u}^{\delta}(t) + \sum_{i} \frac{\delta^\alpha R_{i}^{\delta}(t)}{1 + \delta^\alpha R_{i}^{\delta}(t)} \big( 1 - \bar{u}^{\delta}(t) R_{i}^{\delta}(t) \big) \frac{\delta^\alpha}{|x-x_i|}
\end{equation}
for $x \in \Omega \setminus \cup B_{i}^{\delta} (t)$.

Below is a maximum principle tailored to our setting that will be used to compare the approximation and the solution in the lemma next. Its proof can be found in \cite{niethammer}.
\begin{lemma}\label{maximum_principle}
Let $\Omega$ be a Lipschitz domain and let $\cup B_i \subset \Omega$ be a finite collection of disjoint closed balls. Then, a function $v$ which is constant on each of the boundaries $\partial B_i$ and satisfies 
\[
-\Delta v = 0\quad \text{in } \Omega \setminus \cup B_i,
\]
\[
v - c_i \int_{\partial B_i} \nabla v \cdot \n \geq 0\quad \text{on } \partial B_i,
\]
\[
\nabla v \cdot \n \geq 0\quad \text{on } \partial \Omega,
\]
where $c_i \geq 0$ for all $i$, also satisfies 
\[
v \geq 0\quad \text{in } \Omega \setminus \cup B_i.
\]
\end{lemma}

\begin{lemma}\label{u_bound}
For any time $t\in(0,T)$ and small positive $\varepsilon$, the chemical potential and its approximation satisfy  
\[
\| u^\delta - \zeta^\delta \|_{L^\infty(\Omega \setminus \cup B_{i}^{\delta})}(t) \leq C \delta^{2\alpha -3 -\varepsilon} \sup R_{i}^{\delta}(t) \big( 1+ \bar{u}^\delta(t) \sup R_{i}^{\delta}(t) \big).
\]
\end{lemma}
\begin{proof}
Since the difference $u^\delta - \zeta^\delta$ is already harmonic in $\Omega \setminus \cup B_{i}^{\delta}$ as $\zeta^\delta$ is a superposition of fundamental solutions, we would like to estimate to what extent it satisfies the maximum principle's boundary conditions. 

For the condition on the particle boundaries, we use equations \eqref{p1}, \eqref{p3}, and the definition of $\zeta^\delta$ to calculate for $x$ on the boundary $\partial B_{i}^{\delta}$ of the $i$th particle,
\begin{equation*}
\begin{split}
\bigg| \bigg( \zeta&^\delta(t,x) - \av_{\partial B_{i}^{\delta}} \nabla \zeta^\delta \cdot \n \bigg) 
- \bigg( u^\delta(t,x) - \av_{\partial B_{i}^{\delta}} \nabla u^\delta \cdot \n \bigg) \bigg| \\ 
&= \bigg| \zeta^\delta - \frac{1}{R_{i}^{\delta}} - \av_{\partial B_{i}^{\delta}} \nabla \zeta^\delta \cdot \n \bigg| \\
&= \bigg| \bar{u}^{\delta} - \frac{1}{R_{i}^{\delta}} + \sum_{j}\bigg\{ \frac{\delta^\alpha R_{j}^{\delta}}{1 + \delta^\alpha R_{j}^{\delta}} (1 - \bar{u}^{\delta} R_{j}^{\delta}) \bigg( \frac{\delta^\alpha}{|x-x_j|} - \av_{\partial B_{i}^{\delta}} \nabla \left( \frac{\delta^\alpha}{|x-x_j|} \right) \cdot \n \bigg) \bigg\} \bigg|, 
\end{split}
\end{equation*}
and since by the divergence theorem there holds for $j \neq i$, 
\[
\av_{\partial B_{i}^{\delta}} \nabla \left( \frac{\delta^\alpha}{|x-x_j|} \right) \cdot \n = 0,
\]
while for $j=i$,
\[
\av_{\partial B_{i}^{\delta}} \nabla \left( \frac{\delta^\alpha}{|x-x_i|} \right) \cdot \n = - \frac{1}{\delta^{\alpha}(R_{i}^{\delta})^2},
\]
we continue the calculation to get
\begin{align}
& = \bigg| \bar{u}^{\delta} - \frac{1}{R_{i}^{\delta}} + \frac{1 - \bar{u}^{\delta} R_{i}^{\delta}}{R_{i}^{\delta}(1 + \delta^\alpha R_{i}^{\delta})} + \sum_{j} \frac{\delta^\alpha R_{j}^{\delta}}{1 + \delta^\alpha R_{j}^{\delta}} (1 - \bar{u}^{\delta} R_{j}^{\delta}) \frac{\delta^\alpha}{|x-x_j|} \bigg| \nonumber\\
& = \bigg| \sum_{j \neq i} \frac{\delta^\alpha R_{j}^{\delta}}{1+ \delta^\alpha R_{j}^{\delta}} 
( 1 - \bar{u}^{\delta} R_{j}^{\delta} ) \frac{\delta^\alpha}{|x-x_j|} \bigg| \nonumber\\ 
& \leq \delta^{2\alpha-3} \sup R_{j}^{\delta} (1 + \bar{u}^{\delta} \sup R_{j}^{\delta} ) \sum_{j \neq i} \frac{\delta^3}{|x-x_j|} \nonumber \\
& \leq C \delta^{2\alpha-3} \sup R_{j}^{\delta} (1 + \bar{u}^{\delta} \sup R_{j}^{\delta} ).\label{est1} 
\end{align}
In the last step, keeping in mind the assumptions on the spatial distribution of particle centers, the sum is bounded for $j \neq i$ since it is considered as a Riemann-sum approximation to the integral 
\[
\int_\Omega \frac{1}{|x-y|}\, \dd y,
\]
which in turn is bounded using radial symmetry around the singularity and where the factor $\delta^3$ in the sum compensates for the scaling in space.

To further fulfil the maximum principle's outer boundary condition on $\partial \Omega$, we consider the comparison function $\zeta^\delta + z^\delta$, where the auxiliary function $z^\delta$ solves the problem 
\begin{equation}
\left.
\begin{array}{cl}
-\Delta z^\delta = \displaystyle{\int_{\partial \Omega}} \nabla \zeta^\delta \cdot \n &\text{in } \Omega, \medskip\\ 
\nabla z^\delta \cdot \n = - \nabla \zeta^\delta \cdot \n &\text{on } \partial \Omega, \medskip\\
\displaystyle{\int_\Omega} z^\delta = 0, & 
\end{array}\right\}
\end{equation}
such that the comparison function $\zeta^\delta + z^\delta$ has zero normal derivative on $\partial \Omega$. To work with the maximum principle, $z^\delta$ also needs to be harmonic in $\Omega$ and for that we need that the integral $\int_{\partial \Omega} \nabla \zeta^\delta \cdot \n$ vanishes. But, 
\[
\int_{\partial \Omega} \nabla \zeta^\delta \cdot \n = 
\delta^\alpha \sum_{i} \frac{\delta^\alpha R_{i}^{\delta}}{1+ \delta^\alpha R_{i}^{\delta}} 
( 1 - R_{i}^{\delta} \bar{u}^{\delta} ) \int_{\partial \Omega} \nabla \left( \frac{1}{|x-x_i|} \right) \cdot \n,\]
where the last integral equals $-4\pi$, independent of $i$. Thus, $z^\delta$ 
is harmonic if and only if
\[
\bar{u}^{\delta} = \sum_i \frac{R_{i}^{\delta}}{1+ \delta^\alpha R_{i}^{\delta}}
\bigg/ \sum_i \frac{(R_{i}^{\delta})^2}{1+ \delta^\alpha R_{i}^{\delta}},
\]
which is exactly the mean field \eqref{meanfield} as dictated by the single-particle ansatz in the beginning of the section. Moreover, since now $z^\delta$ is harmonic, the divergence theorem further gives
\[
\int_{\partial B_{i}^{\delta}} \nabla z^\delta \cdot \n = 0.
\]

A construction as in Lemma 3 of \cite{niethammer1} and elliptic regularity theory (see Gilbarg and Trudinger \cite{gilbarg-trudinger}) give the estimate
\[
\| z^\delta \|_{L^\infty(\Omega)} \leq C_\varepsilon \delta^{2\alpha -3 -\varepsilon} 
\sup R_{i}^{\delta} ( 1+ \bar{u}^\delta \sup R_{i}^{\delta} ),\]
where $\varepsilon$ is a small positive number.

Let us now apply the maximum principle to the function 
\[
f_+ := u^\delta - \zeta^\delta - z^\delta + 
C \delta^{2\alpha -3 -\varepsilon} \sup R_{i}^{\delta} (1+\bar{u}^\delta \sup R_{i}^{\delta}).
\]
For a large enough constant $C$, say $2C_\varepsilon$, the following hold for
$f_+$: it is harmonic, there holds $\nabla f_+ \cdot \n = 0$ on $\partial
\Omega$ by the construction of $z^\delta$, and for the constants $c_i =
1/ {4 \pi \delta^{2\alpha} (R_{i}^{\delta})^2}$, estimate \eqref{est1}
gives 
\[
u^\delta -\zeta^\delta - z^\delta +C \delta^{2\alpha -3 -\varepsilon} 
\sup R_{i}^{\delta} ( 1+ \bar{u}^\delta \sup R_{i}^{\delta} ) - c_i \int_{\partial B_{i}^{\delta}} 
\nabla ( u^\delta - \zeta^\delta - z^\delta ) \cdot \n \geq 0.
\]
Thus, $f_+$ satisfies the maximum principle's conditions and therefore, $f_+ \geq 0$ in $\Omega \setminus \cup B_{i}^{\delta}$, i.e., 
\[
u^\delta - \zeta^\delta - z^\delta \geq -C \delta^{2\alpha -3 -\varepsilon} \sup R_{i}^{\delta} ( 1+ \bar{u}^\delta \sup R_{i}^{\delta} ).
\]
Using the maximum principle with $-v$ instead of $v$, the function
\[
f_- := u^\delta - \zeta^\delta - z^\delta - C \delta^{2\alpha -3 -\varepsilon} \sup R_{i}^{\delta} (1+ \bar{u}^\delta \sup  R_{i}^{\delta} ),
\]
again satisfies the corresponding conditions and, as above, yields $f_- \leq 0$ in $\Omega \setminus \cup B_{i}^{\delta}$, i.e., 
\[
u^\delta - \zeta^\delta - z^\delta \leq C \delta^{2\alpha -3 -\varepsilon} \sup R_{i}^{\delta} ( 1+\bar{u}^\delta \sup R_{i}^{\delta} ).
\]

Combining the last two inequalities, we get 
\[
\| u^\delta - \zeta^\delta - z^\delta \|_{L^\infty(\Omega \setminus \cup B_{i}^{\delta})} \leq C \delta^{2\alpha -3 -\varepsilon} \sup R_{i}^{\delta} ( 1+ \bar{u}^\delta \sup R_{i}^{\delta} )
\]
and the lemma follows by the triangle inequality using the regularity of $z^\delta$.
\end{proof}

In the previous lemma it is clear that our approach excludes the critical case
$\alpha=3 / 2$. In the following we introduce, for technical reasons, a new
exponent $\gamma >0$ with the property
\[
\delta^\gamma := \max\, \{ \delta^\alpha, \delta^{2\alpha -3}, \delta^{2\alpha -3 -\varepsilon} \}
\]
for each $\alpha$ greater than $3/2 +\varepsilon$.

As a corollary to the previous lemma we can now estimate the effect of the mean field. 
\begin{corollary}\label{uu_bound}
For any time $t\in(0,T)$ and $\gamma >0$, the chemical potential and the mean field satisfy 
\begin{align*}
\| u^\delta - \bar{u}^{\delta} \|_{L^\infty(\Omega \setminus \cup B_{i}^{\delta})}(t) \leq 
C \delta^\gamma \big( 1+ 2 \sup R_{i}^{\delta}(t) \big) \big(1+ \bar{u}^\delta(t) \sup R_{i}^{\delta}(t)\big).
\end{align*}
\end{corollary}
\begin{proof}
By the triangle inequality and Lemma \ref{u_bound} there holds
\begin{equation*}
\| u^\delta - \bar{u}^{\delta} \|_{L^\infty(\Omega \setminus \cup B_{j}^{\delta})} \leq 
\| \zeta^\delta -\bar{u}^\delta \|_{L^\infty(\Omega \setminus \cup B_{j}^{\delta})}  
+ C \delta^{2\alpha -3 -\varepsilon} \sup R_{j}^{\delta} ( 1+ \bar{u}^\delta \sup R_{j}^{\delta} ).
\end{equation*}
To estimate $\| \zeta^\delta -\bar{u}^\delta \|_{L^\infty(\Omega \setminus \cup B_{j}^{\delta})}$, by the definition of $\zeta^\delta$ there holds for $x \in \Omega \setminus \cup B_{j}^{\delta}$,
\begin{align*}
& \Big| \zeta^\delta (t,x) - \bar{u}^\delta (t) \Big| = \bigg| \sum_{j} \frac{\delta^\alpha R_{j}^{\delta}}{1+ \delta^\alpha R_{j}^{\delta}} ( 1 - \bar{u}^{\delta} R_{j}^{\delta} ) \frac{\delta^\alpha}{|x-x_j|} \bigg| \\
& \quad \leq \frac{\delta^\alpha R_{i}^{\delta}}{1+ \delta^\alpha R_{i}^{\delta}} ( 1 + \bar{u}^{\delta} R_{i}^{\delta} ) \frac{\delta^\alpha}{|x-x_i|} + \bigg| \sum_{j \neq i} \frac{\delta^\alpha R_{j}^{\delta}}{1+ \delta^\alpha R_{j}^{\delta}} ( 1 - \bar{u}^{\delta} R_{j}^{\delta} ) \frac{\delta^\alpha}{|x-x_j|} \bigg|
\end{align*}
and since $|x-x_i| \geq \delta^\alpha R_{i}^{\delta}$ in $\Omega \setminus \cup B_{j}^{\delta}$, arguing as in estimate \eqref{est1} gives
\begin{align*}
& \kern-75pt \leq \frac{\delta^\alpha ( 1 + \bar{u}^{\delta} R_{i}^{\delta} ) }{1+ \delta^\alpha R_{i}^{\delta}} + C \delta^{2\alpha -3} \sup R_{j}^{\delta} ( 1+ \bar{u}^\delta \sup R_{j}^{\delta} )\\
& \kern-75pt \leq C (\delta^\alpha + \delta^{2\alpha -3} \sup R_{j}^{\delta} ) ( 1+ \bar{u}^\delta \sup R_{j}^{\delta} ),
\end{align*}
thus,
\[
\| \zeta^\delta -\bar{u}^\delta \|_{L^\infty(\Omega \setminus \cup B_{j}^{\delta})} 
\leq C (\delta^\alpha + \delta^{2\alpha -3} \sup R_{j}^{\delta} ) ( 1+ \bar{u}^\delta \sup R_{j}^{\delta} ),
\]
and finally,
\[
\| u^\delta - \bar{u}^{\delta} \|_{L^\infty(\Omega \setminus \cup B_{j}^{\delta})} \leq 
C (  \delta^\alpha + \delta^{2\alpha -3} \sup R_{j}^{\delta} + \delta^{2\alpha -3 -\varepsilon} \sup R_{j}^{\delta} ) (1+ \bar{u}^\delta \sup R_{j}^{\delta} ).
\]
Using the exponent $\gamma$, we get 
\[
\| u^\delta - \bar{u}^{\delta} \|_{L^\infty(\Omega \setminus \cup B_{i}^{\delta})} \leq 
C \delta^\gamma ( 1 + 2 \sup R_{i}^{\delta} ) (1+ \bar{u}^\delta \sup R_{i}^{\delta} ).\qedhere
\]
\end{proof}

The following lemma gives an estimate for the growth rate of particles in accordance with the reaction-controlled Lifshitz--Slyozov--Wagner theory.
\begin{lemma}\label{law_lemma}
For any time $t\in(0,T)$ and $\gamma > 0$, for the growth rates of particles
holds 
\begin{equation*}
\bigg| \dot{R}_{i}^{\delta} - \bigg( \bar{u}^{\delta} - \frac{1}{R_{i}^{\delta}} \bigg) \bigg| 
\leq C \delta^\gamma (1 + \bar{u}^\delta \sup R_{i}^{\delta}) \big( 1+ (1+ \delta^\gamma \sup R_{i}^{\delta})(1+ 2 \sup R_{i}^{\delta}) \big).
\end{equation*}
\end{lemma}
\begin{proof}
Let $w_{i}^{\delta}$ be the capacity potential of the ball $B_{i}^{\delta}$ with respect to a larger ball $B_{i}^{\lambda \delta}:=B(x_i, \lambda \delta^\alpha R_{i}^{\delta})$ for $\lambda > 1$,
i.e., let $w_i$ solve
\begin{equation}\label{capacity}
\left.
\begin{array}{rl}
-\Delta w_{i}^{\delta} = 0 &\text{in } B_{i}^{\lambda \delta} \setminus B_{i}^{\delta},\medskip\\
w_{i}^{\delta} = 0 &\text{on } \partial B_{i}^{\lambda \delta},\medskip\\
w_{i}^{\delta} = 1 &\text{in } B_{i}^{\delta}.
\end{array}\right\}
\end{equation}
An explicit calculation gives 
\begin{equation}
w_{i}^{\delta} = \frac{1}{1-\lambda} \left(1-\frac{\lambda\delta^\alpha R_{i}^{\delta}}{|x-x_i|} \right)
\end{equation}
and also
\begin{equation}\label{cap_prop}
\int_{\partial B_{i}^{\delta}} \nabla w_{i}^{\delta} \cdot \n = \int_{\partial B_{i}^{\lambda \delta}} \nabla w_{i}^{\delta} \cdot \n = 4 \pi \frac{\lambda}{1 - \lambda} \delta^{\alpha} R_{i}^{\delta}.
\end{equation}

Using equations \eqref{p1}, \eqref{p2}, \eqref{p3}, and the Neumann boundary condition, along with the above properties of $w_{i}^{\delta}$, and integrating by parts, gives
\begin{align*}
4\pi\delta^{2\alpha} ({R_{i}^{\delta}})^2 \dot{R}_{i}^{\delta} & = \int_{\partial B_{i}^{\delta}} \nabla u^\delta \cdot \n\\
& = \int_{\partial B_{i}^{\delta}} w_{i}^{\delta} \nabla u^\delta \cdot \n\\
& = -\int_{B_{i}^{\lambda \delta} \setminus B_{i}^{\delta}} \nabla w_{i}^{\delta} \nabla u^\delta\\
& = \int_{\partial B_{i}^{\delta}} u^{\delta} \nabla w_{i}^{\delta} \cdot \n - \int_{\partial B_i^{\lambda \delta}} u^{\delta}  \nabla w_{i}^{\delta} \cdot \n\\
& = \int_{\partial B_{i}^{\delta}} \left( \frac{1}{R_{i}^{\delta}} + \dot{R}_{i}^{\delta} \right)  \nabla w_{i}^{\delta} \cdot \n - \int_{\partial B_i^{\lambda \delta}} u^{\delta}  \nabla w_{i}^{\delta} \cdot \n\\
& = 4 \pi \frac{\lambda}{1-\lambda} \delta^\alpha R_{i}^{\delta} \left(\frac{1}{R_{i}^{\delta}} + \dot{R}_{i}^{\delta} - \bar{u}^{\delta} \right) - \int_{\partial B_i^{\lambda \delta}} (u^{\delta} - \bar{u}^{\delta})  \nabla w_{i}^{\delta} \cdot \n,
\end{align*}
where in the last equation we used \eqref{cap_prop} and added and subtracted $\bar{u}^{\delta}$. After rearranging, we have
\begin{align}\label{law}
\bigg| \dot{R}_{i}^{\delta} - \bigg( \bar{u}^{\delta} - \frac{1}{R_{i}^{\delta}} \bigg) \bigg| 
& =\bigg| \frac{1-\lambda}{\lambda} \delta^\alpha R_{i}^{\delta}\dot{R}_{i}^{\delta} +
\frac{1-\lambda}{4 \pi \lambda \delta^\alpha R_{i}^{\delta}} 
\int_{\partial B_{i}^{\lambda \delta}} ( u^\delta - \bar{u}^{\delta} ) \nabla w_{i}^{\delta} \cdot \n \bigg|\nonumber\\
& \leq \frac{\lambda - 1}{\lambda} \delta^\alpha R_{i}^{\delta} | \dot{R}_{i}^{\delta} | 
+ \frac{\lambda - 1}{4 \pi \lambda \delta^\alpha R_{i}^{\delta}} \,
\bigg| \int_{\partial B_{i}^{\lambda \delta}} ( u^\delta - \bar{u}^{\delta} ) \nabla w_{i}^{\delta} \cdot \n \bigg| \nonumber\\
& \leq \delta^\alpha R_{i}^{\delta} | \dot{R}_{i}^{\delta} | 
+ \| u^\delta - \bar{u}^{\delta} \|_{L^\infty(\Omega \setminus \cup B_{i}^{\delta})},
\end{align}
where in the last step we again used equation \eqref{cap_prop}. But by using equation \eqref{p3} for $u^\delta$ on $\partial B_{i}^{\delta}$ we have
\begin{equation}\label{rr_bound}
 R_{i}^{\delta} | \dot{R}_{i}^{\delta} | \leq 1 + R_{i}^{\delta} |u^\delta| \leq 1 + R_{i}^{\delta} (\| u^\delta - \bar{u}^{\delta}\|_{L^\infty(\Omega \setminus \cup B_{i}^{\delta})} + \bar{u}^{\delta}).
 \end{equation}
Substituting back in \eqref{law} and using Corollary \ref{uu_bound} gives the
final estimate.
\end{proof}

The next lemma ensures that the bounds in the approximation and the growth-rate estimates are indeed uniform.
\begin{lemma}\label{gronwall}
For any time $t\in(0,T)$, the mean field and the radii of the particles are uniformly bounded, i.e.,
\[
\bar{u}^{\delta}(t) \leq C ~~\text{ and }~~ \sup R_{i}^{\delta}(t)  \leq C.
\]
\end{lemma}
\begin{proof}
For the mean field \eqref{meanfield} holds 
\[
\bar{u}^\delta = \sum_i \frac{R_{i}^{\delta}}{1+\delta^\alpha R_{i}^{\delta}} \bigg/ \sum_i \frac{(R_{i}^{\delta})^2}{1+ \delta^\alpha R_{i}^{\delta}}
\leq \sup R_{i}^{\delta} (1 + \delta^\alpha \sup R_{i}^{\delta}) \sum_i R_{i}^{\delta} \bigg/ \sum_i (R_{i}^{\delta})^{3}
\] 
and since by H\"older's inequality
\[
\sum_i R_{i}^{\delta} \bigg/ \sum_i (R_{i}^{\delta})^{3} \leq \bigg( \sum_i 1 \bigg/ \sum_i
(R_{i}^{\delta})^{3} \bigg)^{2/3} \leq \bigg( 1 \bigg/ \sum_i \delta^3
(R_{i}^{\delta})^{3} \bigg)^{2/3},
\]
conservation of the total volume of particles gives
\[
\bar{u}^{\delta} \leq C \sup R_{i}^{\delta} (1 + \delta^\alpha \sup R_{i}^{\delta}).
\]
or, using the exponent $\gamma$,
\begin{equation}\label{ubar_estimate}
\bar{u}^{\delta} \leq C \sup R_{i}^{\delta} (1 + \delta^\gamma \sup R_{i}^{\delta}).
\end{equation}

Consider now the set 
\[
A:= \left\{ t \mid \sup R_{i}^{\delta}(t) \leq \frac{1}{\delta^{\gamma / 4}}
\right\};
\]
then, for times $t\in A$, plugging \eqref{ubar_estimate} in estimate \eqref{rr_bound} and using Corollary \ref{uu_bound} gives
\[
\frac{\dd}{\dd t} (R_{i}^{\delta})^2 \leq C \sup (R_{i}^{\delta})^2 + C.
\]

Integrating over the time interval $(0,T)$, Gronwall's inequality implies that
\[
{\textstyle \sup_i}\, {\textstyle \sup_{t\in A \cap [0,T]}}\, (R_{i}^{\delta})^2 \leq C(T),
\]
therefore, $[0,T] \subset A$, i.e., the radii are bounded up to time $T$ as is the mean field by estimate \eqref{ubar_estimate}.
\end{proof}

Finally, the following lemma gives control over the growth rates of vanishing particles and will prove useful for some regularity considerations in the next section.
\begin{lemma}\label{reg_ineq}
For any time $t \in (0,T)$ such that $R_{i}^{\delta}(t) \leq 1 / {4 \sup_{t,
\delta} \bar{u}^{\delta}(t)}$ and for suf\-fi\-cientl\-y small $\delta$, there
holds 
\[
-\frac{2}{R_{i}^{\delta}} \leq \dot{R}_{i}^{\delta} \leq -\frac{1}{2 R_{i}^{\delta}} < 0
\]
and
\[
\sqrt{t_{i} -t} \leq R_{i}^{\delta} \leq 2\sqrt{t_{i} -t}.
\]
\end{lemma}
\begin{proof}
For 
\[
g := \frac{\delta^\alpha R_{i}^{\delta}}{1 + \delta^\alpha R_{i}^{\delta}} \frac{\delta^\alpha}{|x - x_i|},
\]
it can be verified that the function $u^\delta - g$ satisfies the assumptions of
the maximum principle in Lemma \ref{maximum_principle}  for the constants $c_i =
{1}/{4 \pi \delta^{2\alpha} (R_{i}^{\delta})^2}$, thus yielding $u^\delta
\geq g$ in $\Omega \setminus \cup B_{i}^{\delta}$. But since $u^\delta = g$ on
the boundary $\partial B_{i}^{\delta}$, monotonicity implies that $\nabla
u^\delta \cdot \n \geq \nabla g \cdot \n$ on $\partial B_{i}^{\delta}$ and
taking the average integrals over $\partial B_{i}^{\delta}$ we have
\[ 
\dot{R}_{i}^{\delta} \geq - \frac{1}{R_{i}^{\delta}(1 + \delta^\alpha R_{i}^{\delta})} \geq -\frac{2}{R_{i}^{\delta}}.
\]

Moreover, Lemma \ref{law_lemma} gives 
\[
\dot{R}_{i}^{\delta} \leq \bar{u}^{\delta} - \frac{1}{R_{i}^{\delta}} + C \delta^\gamma (1 + \bar{u}^\delta \sup R_{i}^{\delta}) \big( 1+ (1+ \delta^\gamma \sup R_{i}^{\delta})(1+ 2 \sup R_{i}^{\delta}) \big).
\]
Using now the assumption that $R_{i}^{\delta} \leq {1}/{4 \sup_{t, \delta}
\bar{u}^{\delta}}$ and since from Lemma \ref{gronwall} it follows that  for
sufficiently small $\delta$ the $\mathrm{O}(\delta^\gamma)$ term is uniformly
bounded by ${1}/{4 R_{i}^{\delta}}$, we get
\[
\dot{R}_{i}^{\delta} \leq \frac{1}{4 R_{i}^{\delta}} - \frac{1}{R_{i}^{\delta}} + \frac{1}{4 R_{i}^{\delta}} \leq - \frac{1}{2 R_{i}^{\delta}}.
\]

Let now 
\[
y_1 := \sqrt{t_i - t},\quad y_2 := 2\sqrt{t_i - t}
\]
be sub- and supersolutions that respectively solve
\[
\dot{y}_1 = - \frac{1}{2y_1},\quad \dot{y}_2 = - \frac{2}{y_2}.
\]
By comparison, we get the lemma's second assertion, i.e., $y_1 \leq R_{i}^{\delta} \leq y_2.$
\end{proof}

\section{Homogenization}\label{homogenization}
In order to pass to the homogenization limit of infinitely-many particles, we need 
first describe the particle-radius density in the limit. To that end, define 
at any time $t\in(0,T)$ the empirical measure $\nu_{t}^{\delta}$ as 
\[
\langle\phi,\nu_{t}^{\delta}\rangle = \int \phi\big(t, R_{i}^{\delta}(t) \big)\, \dd \nu_{t}^{\delta} := \frac{1}{N_{\mathrm{i}}} \sum \phi\big(t, R_{i}^{\delta}(t) \big)\quad \text{for }  \phi \in C_{\mathrm{c}},
\]
i.e., for functions $\phi(t,R)$ continuous and compactly supported in the radius variable.

Using now the estimates from the previous section, we can prove the following
\begin{lemma}\label{limits}
For a subsequence $\delta \to 0$ and for a function $\bar{u} \in W^{1,\,p}(0,T)$, for $p<2$, holds 
\begin{align*}
\bar{u}^{\delta} & \to \bar{u}\quad \text{in } L^2(0,T), \\
u^{\delta} & \to \bar{u}\quad \text{in } L^2\big( 0,T; H^1(\Omega) \big).
\end{align*}
Furthermore, the measures $\nu^{\delta}_{t}$ converge to a family $\nu_t$ of probability measures such that 
\[
\int \phi\, \dd \nu_{t}^{\delta} \to \int \phi\, a(t)\, \dd \nu_t\quad \text{uniformly in } t,
\]
where $a(t)$ denotes the percentage of active particles in the limit.
\end{lemma}
\begin{proof}
As a consequence of Lemma \ref{reg_ineq}, we have 
\[
\sup \|\dot{R}_{i}^{\delta} \|_{L^p(0,T)} \leq C(p)\quad \text{for } p < 2,
\]
thus, conservation of volume and boundedness of the radii give for $p<2$,
\[
\bigg\|\frac{\dd}{\dd t} \bar{u}^{\delta} \bigg\|_{L^p(0,T)} \leq C \sup \| \dot{R}_{i}^{\delta} \|_{L^p(0,T)} \leq C.
\]
Therefore, $\bar{u}^{\delta} \in W^{1,\,p}(0,T)$, for $p<2$, and the compactness following from the Rellich--Kondrachov theorem gives that $\bar{u}^{\delta}$ converges to a limit $\bar{u}$ in $L^2$.

Taking into consideration that $u^\delta$ is extended to the whole of $\Omega$ and using the lemmas in the previous section, $\zeta^\delta$ converges to $\bar{u}$ in $L^2(\Omega)$ and $u^\delta - \zeta^\delta$ converges uniformly to $0$ as $\delta \to 0$, therefore, $u^\delta$ converges to $\bar{u}$ in $L^2(\Omega)$. By the energy equality in Lemma \ref{a_priori} we have further control over $\|\nabla u^\delta \|_{L^2}$ and thus, we have strong convergence in $L^2(0,T;H^1(\Omega))$.

For the measures $\nu_{t}^{\delta}$ holds 
\[
\| \nu_{t}^{\delta} \| := \textstyle{\sup_{\| \phi \|_{C_{\mathrm{c}}}\leq 1}} | \langle \phi , \nu_{t}^{\delta} \rangle | \leq 1
\]
in the norm of $(C_{\mathrm{c}})^*$, so for a subsequence $\delta \to 0$ there holds that 
$\nu^{\delta}_{t}$ converges weakly-* to $\nu_t$. Furthermore, for positive functions $\phi$ the limit measure $\nu_t$ is nonnegative and from this it follows that $\nu_t$ becomes zero if there are no particles left in the system.

Choosing now a function $\psi(t)$ that depends only on time, we calculate
\[
\int \psi(t)\, \dd \nu_{t}^{\delta} = \frac{1}{N_{\mathrm{i}}} \psi(t) \sum 1 = \frac{N(t)}{N_{\mathrm{i}}} \psi(t).
\]
The ratio ${N}/{N_{\mathrm{i}}}$ is the percentage of active particles at
time $t$. This ratio is bounded by $1$ and decreasing, therefore it is uniformly
bounded in the space $BV(0,T)$ and by the compact embedding of $BV(0,T) \cap
L^\infty(0,T)$ in $L^2(0,T)$, it converges in $L^2$, for a subsequence $\delta
\to 0$, to a limit $a \in BV(0,T)$.

If we project now the measure $\nu_t$ to the interval $[0,T]$, we get that the projection satisfies
\[
\proj_{[0,\,T]} \nu_t = a(t)\, \dd t
\]
and according to \cite[Ch.~1, Thm.~10]{evans}, the decomposition  and convergence to $\nu_t$ follow from the slicing of measures.
\end{proof}

We conclude with the following theorem which states that the limit measure $\nu_t$ satisfies 
the Lifshitz--Slyozov--Wagner equation in a weak sense. Note that in the theorem's statement, the initial condition is defined as
\[
\int \phi\big(t, R_{i}^{\delta}(t) \big)\, \dd \nu_{0}^{\delta} := \frac{1}{N_{\mathrm{i}}} \sum \phi\big(0, R_{i}^{\delta}(0) \big).
\]

\begin{theorem}
The measure $\nu_t$ satisfies the Lifshitz--Slyozov--Wagner equation in the sense that 
\begin{equation}\label{weak_lsw}
\int \left\{ \frac{\partial}{\partial t} \phi (t,R) + \left( \bar{u} - \frac{1}{R} \right) \frac{\partial}{\partial R} \phi (t,R) \right\}\, a(t)\, \dd \nu_t + \int \phi(0,R)\, \dd \nu_0 = 0,
\end{equation}
for all smooth and compactly supported functions $\phi \in C_{\mathrm{c}}^{\infty} ([0,T) \times \R_+ )$, where the mean field $\bar{u}$ is given by
\[
\bar{u} = \int R\, \dd \nu_t  \,\bigg/  \int R^2\, \dd \nu_t.
\]
\end{theorem}
\begin{proof}
We begin by computing the mean-field limit $\bar{u} (t)$. For a continuous function $\phi(t)$ there holds, by the definition of $\bar{u}^\delta$,
\begin{align*}
\int \phi \bar{u}^{\delta} \frac{R^2}{1+\delta^\alpha R} \dd \nu_{t}^{\delta} &= 
\frac{1}{N_{\mathrm{i}}} \phi \bar{u}^{\delta} \sum \frac{R_{i}^{2}}{1+\delta^\alpha R_i} \\ 
& =\frac{1}{N_{\mathrm{i}}} \phi \sum \frac{R_{i}}{1+\delta^\alpha R_i}\\ 
& = \int \phi \frac{R}{1+ \delta^\alpha R}\, \dd \nu_{t}^{\delta}.
\end{align*}
Taking the limit $\delta \to 0$ on both sides, Lemma \ref{limits} gives 
\[
\bar{u} = \int R\, \dd \nu_t \,\bigg/ \int R^2\, \dd \nu_t.
\]

Consider now a smooth and compactly supported function $\phi$ as in the theorem's statement. Then, the fundamental theorem of calculus and Lemma \ref{law_lemma} give 
\begin{align*}
0 &= \frac{\dd}{\dd t} \int \phi \big( t, R_{i}^{\delta}(t) \big)\, \dd \nu_{t}^{\delta} + 
\int \phi \big( 0, R_{i}^{\delta}(0) \big)\, \dd \nu_{0}^{\delta}  \\
& = \int \bigg\{ \frac{\partial}{\partial t} \phi\big( t, R_{i}^{\delta}(t) \big) + \dot{R}_{i}^{\delta}(t) \frac{\partial}{\partial R} \phi\big( t, R_{i}^{\delta}(t) \big) \bigg\}\, \dd \nu_{t}^{\delta} + \int \phi \big( 0, R_{i}^{\delta}(0) \big)\, \dd \nu_{0}^{\delta} \\
& = \int \bigg\{ \frac{\partial}{\partial t} \phi\big( t, R_{i}^{\delta}(t) \big) + \bigg(\bar{u}^{\delta} - \frac{1}{R_{i}^{\delta}}\bigg) \frac{\partial}{\partial R} \phi\big( t, R_{i}^{\delta}(t) \big)  \bigg\}\, \dd \nu_{t}^{\delta} + \mathrm{O}(\delta^{\gamma}) \\ 
&\quad  + \int \phi \big( 0, R_{i}^{\delta}(0) \big)\, \dd \nu_{0}^{\delta}. 
\end{align*}
The result follows by taking the limit for a subsequence $\delta \to 0$ and using the strong convergence of $\bar{u}^{\delta}$.
\end{proof}

As a concluding remark, we note that the well-posedness (existence, uniqueness, and continuous dependence on initial data) of the weak formulation \eqref{weak_lsw} can be treated by the methods developed by Niethammer and Pego in \cite{np1} and \cite{np2}.

\section*{Acknowledgments}
Thanks are due to Barbara Niethammer for her substantial help and to Nick Alikakos and Bob Pego for helpful discussions. Thanks are also due to the anonymous referee for a careful reading of the manuscript.

\nocite{*}
\bibliographystyle{plain}

\begin{thebibliography}{10}
\bibitem{alikakos1}
N.~D.~Alikakos and G.~Fusco.
\newblock The equations of {O}stwald ripening for dilute systems.
\newblock {\em J.\ Stat.\ Phys.} {\bf 95} No.~5/6 (1999), pp.~851--866.

\bibitem{alikakos2}
N.~D.~Alikakos and G.~Fusco.
\newblock Ostwald ripening for dilute systems under quasistationary dynamics.
\newblock {\em Comm.\ Math.\ Phys.} {\bf 238} No.~3 (2003), pp.~429--479.

\bibitem{bartelt}
N.~C.~Bartelt, W.~Theis, and R.~M.~Tromp.
\newblock Ostwald ripening of two-dimensional islands on Si(001).
\newblock {\em Phys.\ Rev.\ B} {\bf 54} (1996), pp.~11741--11751.

\bibitem{murat}
D.~Cioranescu and F.~Murat.
\newblock A strange term coming from nowhere.
\newblock In {\em Topics in the mathematical modelling of composite materials}, A.~Cherkaev, R.~Kohn eds. Birkh\"auser, Boston, MA, 1997, pp.~45--94.

\bibitem{daipego}
S.~Dai and R.~L.~Pego.
\newblock Universal bounds on coarsening rates for mean-field models of phase transitions.
\newblock {\em SIAM J.\ Math.\ Anal.} {\bf 37} No.~2 (2005), pp.~347--371.

\bibitem{evans}
L.~C.~Evans.
\newblock {\em Weak convergence methods for nonlinear partial differential equations}.
\newblock American Mathematical Society, Providence, RI, 1990.

\bibitem{gilbarg-trudinger}
D.~Gilbarg and N.~S.~Trudinger.
\newblock {\em Elliptic partial differential equations of second order}.
\newblock Springer-Verlag, Berlin, second edition, 1983.

\bibitem{gurtin}
M.~E.~Gurtin.
\newblock {\em Thermomechanics of evolving phase boundaries in the plane}.
\newblock The Clarendon Press, Oxford, 1993.

\bibitem{ls}
I.~M.~Lifshitz and V.~V.~Slyozov.
\newblock The kinetics of precipitation from supersaturated solid solutions.
\newblock {\em J.\ Phys.\ Chem.\ Solids} {\bf 19} (1961), pp.~35--50.

\bibitem{ms}
W.~W.~Mullins and R.~F.~Sekerka.
\newblock Morphological stability of a particle growing by diffusion or heat flow.
\newblock {\em J.\ Appl.\ Phys.} {\bf 34} No.~2 (1963), pp.~323--329.

\bibitem{niethammer1}
B.~Niethammer.
\newblock Derivation of the {LSW}-theory for {O}stwald ripening by homogenization methods.
\newblock {\em Arch.\ Rat.\ Mech.\ Anal.} {\bf 147} (1999), pp.~119--178.

\bibitem{niethammer}
B.~Niethammer.
\newblock The {LSW} model for {O}stwald ripening with kinetic undercooling.
\newblock {\em Proc.\ R.\ Soc.\ Edinburgh} {\bf 130A} No.~6 (2000), pp.~1337--1361.

\bibitem{niet-otto}
B.~Niethammer and F.~Otto.
\newblock {O}stwald ripening: The screening length revisited.
\newblock {\em Calc.\ Var.} {\bf 13} No.~1 (2001), pp.~33--68.

\bibitem{np1}
B.~Niethammer and R.~L.~Pego.
\newblock On the initial-value problem in the Lifshitz--Slyozov--Wagner theory of {O}stwald ripening.
\newblock {\em SIAM J.\ Math.\ Anal.} {\bf 31} No.~3 (2000), pp.~467--485.

\bibitem{np2}
B.~Niethammer and R.~L.~Pego.
\newblock Well-posedness for measure transport in a family of nonlocal domain coarsening models.
\newblock {\em Indiana Univ.\ Math.\ J.} {\bf 54} No.~2 (2005), pp.~499--530.

\bibitem{nietvelaz1}
B.~Niethammer and J.~J.~L.~Vel\'azquez.
\newblock Homogenization in coarsening systems I: Deterministic case.
\newblock {\em Math.\ Mod.\ Meth.\ Appl.\ Sci.} {\bf 14} No.~8 (2004), pp.~1211--1233.

\bibitem{nietvelaz2}
B.~Niethammer and J.~J.~L.~Vel\'azquez.
\newblock Homogenization in coarsening systems II: Stochastic case.
\newblock {\em Math.\ Mod.\ Meth.\ Appl.\ Sci.} {\bf 14} No.~9 (2004), pp.~1--24.

\bibitem{voorhees_book}
L.~Ratke and P.~W.~Voorhees.
\newblock {\em Growth and coarsening: {O}stwald ripening in material processing}.
\newblock Springer-Verlag, Berlin, 2002.

\bibitem{slezov}
V.~V.~Slezov and V.~V.~Sagalovich.
\newblock Diffusive decomposition of solid solutions.
\newblock {\em Sov.\ Phys.\ Usp.} {\bf 30} No.~1 (1987), pp.~23--45.

\bibitem{velazquez}
J.~J.~L.~Vel\'azquez.
\newblock On the effect of stochastic fluctuations in the dynamics of the {L}ifshitz--{S}lyozov--{W}agner model.
\newblock {\em J.\ Stat.\ Phys.} {\bf 99} No.~1/2 (2000), pp.~231--252.

\bibitem{voorhees}
P.~W.~Voorhees.
\newblock The theory of {O}stwald ripening.
\newblock {\em J.\ Stat.\ Phys.} {\bf 38} No.~1/2 (1985), pp.~231--252.

\bibitem{wagner}
C.~Wagner.
\newblock Theorie der {A}lterung von {N}iederschl\"agen durch {U}ml\"osen.
\newblock {\em Z.\ Elektrochem.} {\bf 65} No.~7/8 (1961), pp.~581--591.
\end{thebibliography}

\end{document}